\documentclass[11pt, onesided, reqno]{amsart}

\usepackage{amssymb}
\usepackage{amsmath}
\usepackage{color}
\usepackage{enumerate}
\usepackage{graphicx}

\evensidemargin\oddsidemargin

\newtheorem{theorem}{Theorem}

\newtheorem{lemma}[theorem]{Lemma}
\newtheorem{assumption}[theorem]{Assumption}

\theoremstyle{definition}

\newtheorem{remark}[theorem]{Remark}

\newcommand{\eqnsection}{
\renewcommand{\theequation}{\thesection.\arabic{equation}}
    \makeatletter
    \csname  @addtoreset\endcsname{equation}{section}
    \makeatother}
\eqnsection

\def\e{\mathbf{e}}
\def\E{\mathbb{E}}

\def\N{\mathbb{N}}

\def\R{\mathbb{R}}

\def\Pb{\mathbb{P}}

\def\F{\mathcal{F}}

\def\L{\mathcal{L}}

\def\i{\text{i}}

\newcommand{\equi}{\mathop{\sim}\limits}
\def\={{\,\;\mathop{=}\limits^{\text{(law)}}\;\,}}

\def\qed{\hfill$\square$}

\allowdisplaybreaks

\makeatletter
\@namedef{subjclassname@2020}{%
  \textup{2020} Mathematics Subject Classification}
\makeatother

\begin{document}

\title[]{On the maximal displacement of some critical branching L\'evy processes with stable offspring distribution }
\author[Christophe Profeta]{Christophe Profeta}

\address{
Universit\'e Paris-Saclay, CNRS, Univ Evry, Laboratoire de Math\'ematiques et Mod\'elisation d'Evry, 91037, Evry-Courcouronnes, France.
 {\em Email} : {\tt christophe.profeta@univ-evry.fr}
  }



\begin{abstract} 
Let $X$ be a critical branching L\'evy process whose offspring distribution is in the domain of attraction of a stable random variable.
We study the tail probability of the maximum location ever reached by a particle in two different situations : first when the underlying L\'evy process $L$ admits moments of order at least two and is not centered, and then when the distribution of $L$ has a regularly varying tail. This work complements some earlier results in which either $L$ was centered or the offspring distribution was assumed to have moments of order three.
\end{abstract}

\maketitle

\section{Statement of the main result}

\subsection{Introduction}
We consider a one-dimensional branching L\'evy process $X$. It is a continuous-time particle system in which the individuals move according to independent L\'evy processes, and split at exponential times of parameter 1 into a random number $\boldsymbol{p}$ of children. \\

More precisely, an initial particle starts at $t=0$ from the point $x=0$ and move accordingly to the law of a L\'evy process $L$. 
After an exponential time with parameter 1, the particle dies and gives birth to a random number $\boldsymbol{p}$ of children, whose lives start at the location of their parent's death. The children then behave independently one from another and follow the same stochastic pattern as their parents : they move according to $L$ and branche at rate 1.\\

%
%

We assume that the offspring distribution $\boldsymbol{p} = (p_k)_{k\geq0}$ is critical, i.e. $\E[\boldsymbol{p}]=1$ : this implies that the branching process $X$ will die out a.s. As a consequence, one may define its overall maximum {\bf M}, that is, the maximum location ever attained by one of the particle. Many papers have been devoted to the study of the asymptotics of {\bf M}, generally under the assumption that $\E[\boldsymbol{p}^3]<+\infty$. Such a problem was first introduced when $L$ is a Brownian motion in \cite{FlSa} to model the propagation of a mutant allele in a population. Several generalizations have then been proposed, either for $\alpha$-stable processes \cite{LaSh, Pro1} or for spectrally negative L\'evy processes \cite{Pro2}. In all these papers, the choice of the offspring distribution plays no real role in the tail asymptotics of $\bf{M}$, as it only appears as a multiplicative constant of its variance. \\

We shall remove here the condition on the moments of $\boldsymbol{p}$ and rather assume that  $\boldsymbol{p}$  is in the domain of attraction of a $\beta$-stable random variable with $\beta\in(1,2)$ :
\begin{equation}\label{hyp2}
\lim_{n\rightarrow+\infty}n^\beta \sum_{k=n}^{+\infty}  p_k = c_\beta>0.
\end{equation}

In this case,   it was recently proven in \cite{HJRS}  that if $L$ is centered and admits moments of order strictly greater than $\frac{2\beta}{\beta-1}$, then
\begin{equation}\label{eq:HJRS}
\Pb\left({\bf M}\geq x\right) \equi_{x\rightarrow +\infty} x^{-\frac{2}{\beta-1}} \times \left(  \frac{(\beta+1)\sigma^2}{c_\beta (\beta-1)\Gamma(2-\beta)}\right)^{\frac{1}{\beta-1}}  
\end{equation}
where $\sigma^2$ denotes the variance of $L$. \\

%


Our purpose in this paper is thus twofold. We shall first look at the case when $L$ is not centered in order to complement Formula (\ref{eq:HJRS}), and then a situation when $L$ no longer admits moments of order 2. In the following, we shall assume that all the processes and random variables are defined on the same probability space $(\Omega, \F, \Pb)$, and we shall denote  by $\Pb$, with an abuse of notation, both the laws of $X$ and $L$ when started from $0$. Also, to avoid trivialities, we always exclude the case where $-L$ is a subordinator (in which case ${\bf M} =0$ a.s.).

\begin{theorem}\label{theo:m}
Let $L$ be a L\'evy process which is not a compound Poisson process. 
\begin{enumerate}
\item If $\E[L_1]>0$, we assume that there exists $\delta^\ast>0$ such that $\E\left[|L_1|^{\frac{\beta}{\beta-1}+\delta^\ast}\right]<+\infty$. Then,
$$\Pb\left({\bf M}\geq x\right) \equi_{x\rightarrow +\infty} x^{-\frac{1}{\beta-1}}\times  \left( \frac{\E[L_1]}{c_\beta \Gamma(2-\beta)}\right)^{\frac{1}{\beta-1}}.  $$
\item If $\E[L_1]<0$, we  assume that there exists $0<\omega<+\infty$ such that $\Psi(\omega)=0$ and that $\Psi$ is analytic in a neighborhood of $\omega$.
Then, 
$$\Pb\left({\bf M}\geq x\right) \equi_{x\rightarrow +\infty} \kappa_\beta \times e^{-\omega t}$$
for some (implicit) constant $\kappa_\beta>0$.
\end{enumerate}
\end{theorem}

Theorem \ref{theo:m} along with Formula (\ref{eq:HJRS}) thus shows that there exist three regimes according to the sign of $\E[L_1]$, as is usually the case for standard L\'evy processes. Note that when $\E[L_1]<0$, the exponential decay in Point (2) is the same as that of a free L\'evy process $L$, see \cite{BeDo}. \\

We now look at a situation when $L$ no longer admits moments of order 2 and we shall thus make the following assumption :
\begin{assumption}\label{assum}
Set $S_t= \sup_{s\leq t} L_s$ and let $\e$ be an exponential random variable of parameter 1 independent from $L$.
\begin{enumerate}
\item  We assume that the asymptotics of $S_\e$ and $L_\e$ are equivalent and regularly varying of order $\alpha\in(0,2)$, i.e. that there exists a slowly varying function $\ell_\alpha$ such that:  
\begin{equation}\label{hyp1}
\Pb(S_\e \geq x) \equi_{x\rightarrow +\infty} \Pb(L_\e \geq x)\equi_{x\rightarrow+\infty}  \ell_\alpha(x)x^{-\alpha}. 
\end{equation}
To simplify, we assume that when $\alpha=1$, the function $\ell_1$ is constant.  
\item We shall also need a control on the negative tail, hence we  assume that there exists a finite constant $C_\alpha>0$ such that:  
\begin{equation}\label{hyp1bis}
\limsup_{x\rightarrow+\infty} \frac{\Pb(L_\e \leq - x)}{\Pb(L_\e \geq x)} \leq C_\alpha.
\end{equation}
\end{enumerate}
\end{assumption}

Under this assumption, we have the following asymptotics for the distribution of ${\bf M}$.
\begin{theorem}\label{theo:2}
Let $L$ be a L\'evy process such that Assumption \ref{assum} holds. Then, the asymptotics of ${\bf M}$ is given by:
$$\Pb({\bf M}\geq x) \equi_{x\rightarrow+\infty}  \ell_\alpha^{\frac{1}{\beta}}(x) \,   x^{-\frac{\alpha}{\beta}}  \times \left(\frac{\beta-1}{c_\beta \Gamma(2-\beta)} \right)^{\frac{1}{\beta}}.  $$
\end{theorem}

\begin{remark}
Note that letting formally $\beta\uparrow 2$, we obtain an asymptotic of order $\frac{\alpha}{2}$ which was the order obtained in \cite{Pro1} when $L$ was an $\alpha$-stable L\'evy process and the offspring distribution was supposed to have moments of order at least 3.

\end{remark}

%

%
%
%
%

\subsection{Comments on the hypotheses}

\noindent
For $\lambda \in \R$,  let us define the Laplace exponent 
$\Psi(\i \lambda) = \ln \E\left[e^{\i \lambda L_1}\right] $ of the L\'evy process $L$ by

$$\Psi( \lambda) = a\lambda +  \frac{\eta^2}{2}\lambda^2 + \int_{\R} \left(e^{ \lambda x}- 1 -  \lambda x 1_{\{|x|<1\}}\right) \nu(dx)$$
where $a \in \R$ is the drift coefficient, $\eta \in \R$ the Gaussian coefficient and the L\'evy measure $\nu$ satisfies $ \int_{\R} (x^2\wedge 1)\,\nu(dx)<+\infty$. With these notations, we have $\E[L_1] =\Psi^\prime(0^+)$ provided the derivative exists. 
%
%
It is known, see \cite{Wil}, that if the tail of the L\'evy measure $\nu$ of $L$ is  regularly varying, i.e.
$$\nu(x,+\infty) \equi_{x\rightarrow+\infty}  \ell_\alpha(x)x^{-\alpha} $$
then for any fixed $t>0$, 
\begin{equation}\label{eq:fixedt}
\Pb(S_t \geq x) \equi_{x\rightarrow +\infty} \Pb(L_t \geq x)\equi_{x\rightarrow+\infty}  t\, \nu(x, +\infty).
\end{equation}
Assumption (\ref{hyp1}) thus supposes that one can integrate this asymptotics in $t$. This is known to be the case if $\alpha\in(1,2)$. Indeed, in this case the random variable $L_t$ admits a finite expectation for every $t\geq0$. Applying \cite[Theorem 2.1]{LiTa}, we deduce since $\e$ is exponentially distributed with parameter 1 that 
$$\Pb(S_\e \geq x) \equi_{t\rightarrow +\infty} \Pb(L_\e \geq x) \equi_{t\rightarrow+\infty}  \E[\e] \nu(x,+\infty)$$
which is the expected formula. Another example is obtained when $L$ is a stable L\'evy process admitting positive jumps. Indeed, in this case, from Bertoin \cite[Chapter VIII, Prop. 4]{Ber}, it is known that there exists a constant $\kappa_\alpha$ such that 
$$\Pb(S_1 \geq x) \equi_{x\rightarrow +\infty} \Pb(L_1 \geq x)\equi_{x\rightarrow+\infty} \kappa_\alpha x^{-\alpha}. $$
Assumption  (\ref{hyp1}) then follows from the scaling property  and Karamata's Tauberian theorem  \cite[Theorem 1.7.6]{BGT}  which states that for $\gamma> 0$ and $f$ a positive and decreasing function :
\begin{equation}\label{eq:taub}
\mathcal{L}[f](\lambda) \mathop{\sim}\limits_{\lambda\rightarrow0}\frac{1}{\lambda^\gamma}\ell\left(\frac{1}{\lambda}\right)
\quad\Longleftrightarrow\quad f(x)\mathop{\sim}\limits_{x\rightarrow+\infty} \frac{1}{\Gamma(\gamma)}x^{\gamma-1}\ell(x),
\end{equation}
where $\mathcal{L}$ denotes the usual Laplace transform on $(0, +\infty)$ and $\ell$ is a slowly varying function.
\\

\subsection{An integral equation}
The proof of both theorems relies on the study of an integral equation satisfied by 
$u(x) := \Pb({\bf M}\geq  x) $ for $x\geq0$.
\begin{lemma}\label{lem:eqS}
The function $u : [0,+\infty) \rightarrow [0,1]$ satisfies the equation
\begin{equation}\label{eq:keyu}
u(x) = \Pb(S_\e\geq x) + \E\left[1_{\{S_\e< x\}}u(x-L_\e)\right] - \E\left[1_{\{S_\e< x\}}F(u(x-L_\e))\right] 
\end{equation}
where the function $F$ is defined by 
\begin{align*}
F(z) := z -1 + \sum_{n=0}^{+\infty} p_n(1-z)^n=\frac{z^2}{2}  \int_0^{1} (1-t)  \sum_{n=2}^{+\infty}  n(n-1)p_{n} (1-tz)^{n-2} dt.
\end{align*}
\end{lemma}
\begin{proof}
We start by applying the Markov property at the first branching event :
\begin{align*}
\Pb({\bf M}< x) &= p_0 \Pb\left(S_\e< x\right) + \sum_{n=1}^{+\infty} p_n\, \Pb\left(S_\e<x,\; L_\e+{\bf M}^{(1)}<x,\ldots,  L_\e+{\bf M}^{(n)}<x \right)
\end{align*}
where the  random variables $({\bf M}^{(n)})_{n\in \N}$ are independent copies of ${\bf M}$, which are also independent of the pair $(L_\e, S_\e)$.
As a consequence, we obtain the integral equation :
\begin{equation}\label{eq:u0}
1-u(x) = p_0 \Pb\left(S_\e< x\right)+ \sum_{n=1}^{+\infty} p_n \, \E\left[1_{\{S_\e<x\}}\; (1- u(x-L_\e))^n\right].
\end{equation}
Plugging into (\ref{eq:u0}) the Taylor expansion
\begin{equation}\label{eq:taylor}
(1-u)^n = 1 -n u +\frac{n(n-1)}{2}u^2 \int_0^1 (1-ut)^{n-2} (1-t) dt
\end{equation}
then yields, since $\E[\boldsymbol{p}]=\sum n p_n =1$, 
$$1-u(x) = \Pb\left(S_\e< x\right)-  \E\left[1_{\{S_\e<x\}}u(x-L_\e)\right] +\E\left[1_{\{S_\e< x\}}F(u(x-L_\e))\right]. $$
This is Equation (\ref{eq:keyu}), after rearranging the terms.
\end{proof}
We gather below some properties of the function $F$.


%

\begin{lemma}\label{lem:F}
The function $F : [0,1] \rightarrow [0,1]$ is increasing and satisfies: 
\begin{enumerate}
\item For all $z\in[0,1]$ : $F(z)\leq z$
\item The function $z\rightarrow z-F(z)$ is increasing on $[0,1]$
\item $F$ has the asymptotics :
$$F(z) \equi_{z\downarrow 0} c_\beta \frac{\Gamma(2-\beta)}{\beta-1}z^{\beta}. $$
\end{enumerate}
\end{lemma}

\begin{proof}
The fact that $F$ is increasing follows from a change of variable :
$$F(z) = \frac{1}{2}  \int_0^{z} (z-s)  \sum_{n=2}^{+\infty}  n(n-1)p_{n} (1-s)^{n-2} ds.$$
Points (1) and (2) follow from the observation that  for $z\in[0,1]$, 
$$z - F(z)=1- \sum_{n\geq0} (1-z)^n p_n\geq 1 -  \sum_{n\geq0}  p_n =0.$$
Finally, the asymptotics given in Point (3) is classic and we refer to \cite[Lemma 3.1]{HJRS} for instance.
\end{proof}

\begin{remark}
Before tackling the proofs, we briefly show how one can heuristically recover Formula (\ref{eq:HJRS}) and the formulae of Theorem \ref{theo:m} starting from Lemma \ref{lem:eqS}.\begin{enumerate}
\item \textbf{Case $\E[L_1]=0$}.
Assume that $\displaystyle \Pb\left({\bf M}\geq x\right) \equi_{x\rightarrow +\infty} a x^{-\gamma}$ for $a,\gamma>0$.
Plugging this asymptotics into (\ref{eq:keyu}) and using Lemma \ref{lem:F}, we obtain, neglecting the terms in $S_\e$,
$$ \E\left[ \frac{a}{(x-L_\e)^\gamma}  \right]-a x^{-\gamma}  \simeq c_\beta \frac{\Gamma(2-\beta)}{\beta-1} \E\left[ \frac{a^\beta}{(x-L_\e)^{\beta\gamma}}  \right].$$
Using a Taylor expansion and the fact that $L$ is centered, this yields
$$  \frac{a}{x^\gamma} \frac{\gamma(\gamma+1)}{2} \E\left[\left(\frac{L_\e}{x}\right)^2\right] \simeq c_\beta \frac{\Gamma(2-\beta) }{\beta-1}\frac{a^\beta}{x^{\beta\gamma}}.$$
By identification of the power of $x$, we see that $\gamma+2 = \beta\gamma$, i.e. $\displaystyle \gamma=\frac{2}{\beta-1}$ and then
$$  \frac{\beta+1}{\beta-1} \E[L_\e^2]= c_\beta \Gamma(2-\beta)a^{\beta-1}$$
i.e. setting $\sigma^2=\E[L_\e^2] = \E[L_1^2]$ since $L$ is centered, we finally obtain
$$a = \left( \frac{(\beta+1)\sigma^2}{c_\beta (\beta-1) \Gamma(2-\beta)} \right)^{\frac{1}{\beta-1}}$$
which is the asymptotics (\ref{eq:HJRS}) obtained in \cite{HJRS}. 
\item \textbf{Case $\E[L_1]>0$}.
Assuming again that $\displaystyle \Pb\left({\bf M}\geq x\right) \equi_{x\rightarrow +\infty} a x^{-\gamma}$ the same heuristic argument  yields 
$$  \frac{a}{x^\gamma} \gamma \E\left[\left(\frac{L_\e}{x}\right)\right] \simeq c_\beta \frac{\Gamma(2-\beta)}{\beta-1} \frac{a^\beta}{x^{\beta\gamma}},$$
i.e, by identification :
$$\Pb\left({\bf M}\geq x\right) \equi_{x\rightarrow +\infty} x^{-\frac{1}{\beta-1}}\times  \left( \frac{\E[L_\e]}{c_\beta \Gamma(2-\beta)}\right)^{\frac{1}{\beta-1}}.  $$
\item \textbf{Case $\E[L_1]<0$}. In this case, we assume that $\Pb\left({\bf M}\geq x\right) \equi_{x\rightarrow +\infty} \kappa e^{-\omega x}$ with $\kappa, \omega>0$. Plugging this asymptotics in (\ref{eq:keyu}), using Lemma \ref{lem:F} and neglecting again the terms in $S_\e$, we obtain :
$$\kappa \E\left[e^{-\omega(x-L_\e)}\right] - \kappa e^{-\omega x} \simeq  c_\beta \frac{\Gamma(2-\beta)}{\beta-1} \kappa^\beta  \E\left[e^{-\beta \omega (x-L_\e)}\right]$$
which requires to hold that $\E[e^{\omega L_\e}]=1$. Note that since $\e$ is independent from $L$, this condition, known as Cram\'er's condition,  is equivalent to $\E[e^{\omega L_1}]=1$, or also to  $\Psi(\omega)=0$ as stated in the Theorem.
\end{enumerate}

\end{remark}

\section{Proof of Theorem \ref{theo:m}}

The proof of Theorem \ref{theo:m} is similar to that of \cite{Pro1}. We set $u(x)=0$ for $x<0$ and  rewrite Equation (\ref{eq:keyu}) under the form 
\begin{equation}\label{eq:usansS}
u(x) =  \E\left[1_{\{L_\e< x\}}u(x-L_\e)\right] - \E\left[1_{\{L_\e< x\}}F(u(x-L_\e))\right] +R(x)
\end{equation}
where the remainder $R$ is given by 
$$R(x) = \Pb(S_\e\geq x)+ \E\left[\left(1_{\{S_\e< x\}} -1_{\{L_\e< x\}} \right)\left( u(x-L_\e)-F(u(x-L_\e))\right)\right] -1_{\{x<0\}}.$$
Note the presence of the indicator function  to take into account that $u$ is null on $(-\infty, 0)$.
\begin{lemma}\label{lem:R}
The remainder $R$ satisfies the following properties :
\begin{enumerate}
\item The function $R$ is negative on $(-\infty,0)$ and positive on $(0,+\infty)$.
\item For every $x\geq0$ :  $R(x)\leq u(x)$.
\item There are the bounds
$$
|R(x) | \leq
\begin{cases}
 \Pb(S_\e\geq x) & \qquad \text{if }x>0\\
\Pb(L_\e< x) &\qquad \text{if }x<0.
\end{cases}
$$
\end{enumerate}
\end{lemma}

\begin{proof}
By definition, for $x<0$, we have since $S_\e \geq 0$ a.s. 
\begin{equation}\label{eq:R<0}
R(x) = -  \E\left[1_{\{L_\e< x\}} \left( u(x-L_\e)-F(u(x-L_\e))\right)\right] <0
\end{equation}
 which is negative from Lemma \ref{lem:F}. Also, for $x>0$, since $u$ is null on $(-\infty,0)$ and $F(0)=0$, 
\begin{align}
\notag R(x) &=\Pb(S_\e\geq x) -    \E\left[ 1_{\{S_\e\geq  x, \, L_\e< x\}} \left(u(x-L_\e)-F(u(x-L_\e)\right)   \right]\\ 
\label{eq:R>0}&=\E\left[ 1_{\{S_\e\geq x\}} \sum_{n\geq 0} (1-u(x-L_\e))^n p_n\right] \geq0
\end{align}
which proves Point (1). To prove Point (2), observe that going back to  (\ref{eq:usansS})
$$u(x) - R(x) = \E\left[1_{\{L_\e< x\}}\left(u(x-L_\e)-F(u(x-L_\e))\right)  \right] \geq0$$
which is positive from Lemma \ref{lem:F}. Finally, the bounds on $R$ are direct consequences of  (\ref{eq:R<0}) and  (\ref{eq:R>0}).
%
%
\end{proof}

We now tackle the proof of Theorem \ref{theo:m}. In both cases $\E[L_1]>0$ and $\E[L_1]<0$, the idea is to transform Formula (\ref{eq:keyR}) into a more tractable equation from which the asymptotics may be obtained.

\subsection{The case $\E[L_1]>0$}

\noindent
Let us come back to  (\ref{eq:usansS}) and take the Fourier transform of both sides. Applying  Lemma 3.4 in \cite{Pro2} to compute the convolution products, we obtain  for $\xi\neq0$ :
$$\F[u](\xi) = \E\left[e^{\i \xi L_\e}\right] \F[u](\xi) - \E\left[e^{\i \xi L_\e}\right] \F[F\circ u](\xi) + \F[R](\xi) $$
where the Fourier transform  of a function $f$ is defined by 
$$\F[ f](\xi) = \int_\R e^{i\xi x} f(x) dx.$$
Note that although we do not know yet that $u$ is integrable, the Fourier transforms of $u$ and $F\circ u$ are nevertheless well-defined for $\xi\neq0$ since both functions are  positive and decreasing on $(0,+\infty)$. Similarly, the Fourier transform of $R$ is well-defined since $R$ is integrable thanks to Point (3) of Lemma \ref{lem:R}. By definition of the characteristic function $\Psi$, we have
$$\E\left[e^{\i \xi L_\e}\right] = \int_0^{+\infty} e^{-t }e^{t\Psi(\i\xi)} dt =\frac{1}{1-\Psi(\i\xi)}$$
which yields the equation
\begin{equation}\label{eq:FFF}
\F[F\circ u -R](\xi) = \Psi(\i \xi) \F[ u -R](\xi).
\end{equation}
Let us next define the running infimum $I_t = \inf_{s\leq t} L_s$ of the L\'evy process $L$. Applying the  Wiener-Hopf factorisation since $L$ is not a compound Poisson process, see \cite[Section 6.4]{Kyp}, we have 
$$\frac{q}{q-\Psi(i\xi)} = \E\left[e^{\i \xi I_{\e_q} }\right] \E\left[e^{\i \xi S_{\e_q} }\right]$$
where $\e_q$ denotes an exponential r.v. of parameter $q$, independent from $L$.
Since $\E[L_1]>0$, the L\'evy process $L$ converges a.s. to $+\infty$. As a consequence, the random variable $I_\infty = \inf_{t\geq0} L_t$ is well-defined, and passing to the limit as $q\downarrow0$, we deduce that 
\begin{equation}\label{eq:fluct}
- \frac{1}{\Psi(i\xi)}  =\E\left[e^{\i \xi I_\infty}\right]  \frac{k}{\kappa(0, -\i \xi)} 
\end{equation}
where $k>0$ is some normalization constant and  $\kappa(0, -\i \xi)$ is the characteristic exponent of the ladder height process $H$ associated to $L$, i.e. 
$$\E\left[e^{\i \xi H_t}\right] = e^{-\kappa(0, -\i \xi) t}.$$
In particular, since $H$ is a subordinator, $\kappa(0, -\i \xi)$ admits the representation 
$$\kappa(0, -\i \xi) = b(-\i \xi) + \int_0^{+\infty} (1-e^{i\xi  y}) \pi(dy) = (-\i \xi) \left(b + \int_0^{+\infty} e^{i\xi  y}  \overline{\pi}(y) dy\right)$$
where $\pi$ denotes the L\'evy measure of $H$, $\overline{\pi}(y) = \pi((y,+\infty))$ and $b$ is some non-negative constant. Note also that multiplying (\ref{eq:fluct}) by $\xi$ and letting $\xi\downarrow 0$ yields the identity
\begin{equation}\label{eq:psi0}
b + \int_0^{+\infty} \overline{\pi}(y) dy =k \Psi^\prime(0^+). 
\end{equation}
Also, since $\E[|L_1|^{\frac{\beta}{\beta-1} +\delta^\ast}]<\infty$,  we deduce from \cite[Th\'eor\`eme 6.2.3]{Vig} that 
$\displaystyle  \int_0^{+\infty}  y^{\frac{1}{\beta-1}+\delta^\ast}\overline{\pi}(y) dy <+\infty$ which implies, since $\overline{\pi}$ is positive and decreasing, that there exists a constant $C_\pi>0$ such that for all $x>0$, 
\begin{equation}\label{eq:maxpi}
\overline{\pi}(x) \leq  C_\pi x^{-\frac{\beta}{\beta-1} -\delta^\ast}. 
\end{equation}
Now, plugging (\ref{eq:fluct}) into (\ref{eq:FFF}) and  integrating by parts, Formula (\ref{eq:FFF}) becomes
$$ k\E\left[e^{\i \xi I_\infty}\right]  \int_\R e^{\i \xi x} \int_x^{+\infty}  \left(F(u(z)) -R(z) \right)dz = \left(b + \int_0^{+\infty}e^{i\xi  y} \overline{\pi}(y) dy\right)  \F[ u -R](\xi).$$
Inverting the Fourier transforms and using the Fubini-Tonelli theorem, we obtain the key equation:
\begin{equation}\label{eq:keyR}
k\int_{x}^{+\infty}  \left(F(u(y)) -R(y) \right) \Pb\left( I_{\infty} \geq x-y\right) dy =  b (u(x) - R(x)) +   \int_{-\infty}^x   \left(u(z) -R(z) \right) \overline{\pi}(x-z) dz. 
\end{equation}
Note that since $u$ is null on $(-\infty,\,0)$, we deduce from Lemma \ref{lem:R} that the last term is the sum of two positive terms
$$\int_{-\infty}^x   \left(u(z) -R(z) \right) \overline{\pi}(x-z) dz = -  \int_{-\infty}^0 R(z)  \overline{\pi}(x-z) dz+\int_{0}^x   \left(u(z) -R(z) \right) \overline{\pi}(x-z) dz.$$
We now study Equation (\ref{eq:keyR}). The main difficulty here is to deal with the remainder $R$ and show that it is negligible with respect to $u$. 
A first estimate is given by Lemma \ref{lem:R}. Indeed, using Etemadi's inequality (see \cite[Theorem 5.11]{KuSc}) and the Markov inequality, we have for $t>0$,
\begin{equation}\label{eq:etemadi}
\Pb(S_t\geq 3x) \leq \Pb\left(\sup_{s\leq t} |L_s|\geq 3x\right) \leq 3 \Pb( |L_t|\geq x) \leq 3 \E\left[|L_t|^{\frac{\beta}{\beta-1}+ \delta^\ast}\right] x^{-\frac{\beta}{\beta-1} - \delta^\ast}.
\end{equation}
Then, from  the independent increments of L\'evy processes and the standard inequality $(a+b)^c\leq 2^{c} (a^c + b^c)$ for $a,b,c>0$ :
\begin{align*}
\E\left[|L_t|^{\frac{\beta}{\beta-1}+ \delta^\ast}\right] &\leq 2^{\frac{\beta}{\beta-1}+ \delta^\ast } \left(\E\left[|L_{t-\lfloor t \rfloor}  |^{\frac{\beta}{\beta-1}+ \delta^\ast}\right] +   \E\left[|L_{\lfloor t \rfloor}  |^{\frac{\beta}{\beta-1}+ \delta^\ast}\right] \right)\\
&\leq  2^{\frac{\beta}{\beta-1}+ \delta^\ast } \left(\E\left[\sup_{s\leq 1} |L_s|^{\frac{\beta}{\beta-1}+ \delta^\ast}\right] +  \lfloor t \rfloor^{\frac{\beta}{\beta-1}+ \delta^\ast}  \E\left[|L_{1}  |^{\frac{\beta}{\beta-1}+ \delta^\ast}\right] \right)
\end{align*}
where $\lfloor t \rfloor$  denotes the integer part of $t$ and where the last term comes from Minkowski inequality. As a consequence, integrating (\ref{eq:etemadi}) against an exponential function, we deduce from Lemma \ref{lem:R} that there exists a  constant $C_R>0$ such that
\begin{equation}\label{eq:asympR1}
\forall x\in \R\backslash\{0\},\qquad |R(x)| \leq C_R |x|^{-\frac{\beta}{\beta-1} - \delta^\ast}.
\end{equation}
The rest of the proof is decomposed in two steps : starting from (\ref{eq:keyR}), we first obtain some crude asymptotics on $u$, which combined with (\ref{eq:asympR1}) will show that $R$ is indeed negligible, and then compute the exact asymptotics.

\subsubsection{First bounds}

\begin{lemma}\label{lem:1b}
There exist two positive constants $\kappa_1, \kappa_2$ such that 
$$\kappa_1\, x^{-\frac{1}{\beta-1}} \leq u(x) \leq\kappa_2\, x^{-\frac{1}{\beta-1}}\qquad \text{ as }x\rightarrow +\infty.$$
\end{lemma}
\noindent
Note that since $\beta \in (1,2)$, this implies that $\int_0^{+\infty} u(x) dx = \E[{\bf M}]$ is finite.

\begin{proof}
We start with the upper bound. Take $A>0$ large enough. From Lemma \ref{lem:F}, we have for $x\geq A$,
\begin{align*}
\int_{x}^{+\infty} F(u(y))  \Pb\left( I_{\infty} \geq x-y\right) dy &=
 x \int_{1}^{+\infty}  F(u(xz))\Pb\left( I_{\infty} \geq x(1-z)\right) dz\\
 & \geq x F(u(2x))  \int_1^2  \Pb\left( I_{\infty} \geq A(1-z)\right) dz \geq  K_0\, x(u(2x))^\beta 
\end{align*}
for some constant $K_0 >0$.  As a consequence, we obtain from (\ref{eq:keyR}), since $R$ is positive on $(0, \infty)$ :
$$k\,K_0\,x(u(2x))^\beta  \leq b u(x)+ \int_0^{x}  u(x-z) \overline{\pi}(z) dz  + k\,x\int_{1}^{+\infty} R(xz) dz + \int_{-\infty}^0 |R(z)| \overline{\pi}(x-z) dz  
.$$
We now set $\gamma(x)=x^{\frac{1}{\beta-1}}u(x)$. Multiplying the above expression by $x^{\frac{1}{\beta-1}}$, we obtain on the right-hand side :
\begin{align*}
x^{\frac{1}{\beta-1}}\int_0^{x} u(x-z) \overline{\pi}(z) dz  &\leq 2^{\frac{1}{\beta-1}}  \int_0^{x}(x-z)^{\frac{1}{\beta-1}}  u(x-z) \overline{\pi}(z) dz + 2^{\frac{1}{\beta-1}}  \int_0^{x} z^{\frac{1}{\beta-1}} u(x-z) \overline{\pi}(z) dz \\
&\leq 2^{\frac{1}{\beta-1}}  \int_0^{x}\gamma(x-z) \overline{\pi}(z) dz + 2^{\frac{1}{\beta-1}}  \int_0^{+\infty} z^{\frac{1}{\beta-1}}\overline{\pi}(z) dz \\
&\leq 2^{\frac{1}{\beta-1}}  \sup_{y\in(0, x)}\gamma(y)  \int_0^{+\infty} \overline{\pi}(z) dz + 2^{\frac{1}{\beta-1}}  \int_0^{+\infty} z^{\frac{1}{\beta-1}}\overline{\pi}(z) dz 
\end{align*}
and, using the bound (\ref{eq:asympR1}), 
\begin{align*}
& k x^{\frac{\beta}{\beta-1}}\int_{1}^{+\infty} R(xy) dy +x^{\frac{1}{\beta-1}}\int_{-\infty}^0 |R(z)| \overline{\pi}(x-z) dz\\
  &\qquad \leq k\, C_R x^{-\delta^\ast} \int_1^{+\infty} y^{-\frac{\beta}{\beta-1}-\delta^\ast} dy   + x^{\frac{1}{\beta-1}} \int_x^{+\infty} |R(x-z)| \overline{\pi}(z) dz\\
&\qquad \leq \frac{ k \,C_R}{\frac{1}{\beta-1}+\delta^\ast} A^{-\delta^\ast} + \sup_{z\in \R} |R(z)|   \int_{0}^{+\infty}  z^{\frac{1}{\beta-1}} \overline{\pi}(z) dz.
\end{align*}
As a consequence, for $x$ large enough, there exist two positive constants $K_1$ and $K_2$ such that
$$K_1 \gamma^\beta(2x) \leq b \gamma(x) +  2^{\frac{1}{\beta-1}} \sup_{y\in(0, x)}\gamma(y)  \int_0^{+\infty} \overline{\pi}(z) dz  + K_2.$$
We now take the supremum on $x$ in $(A,n)$ with $n>A$, 
$$K_1 \sup_{x\in(2A, 2n)}\gamma^\beta(x) \leq   b \sup_{x\in(A, n)}\gamma(x) + 2^{\frac{1}{\beta-1}} \sup_{x\in(0, n)}\gamma(x) \int_0^{+\infty} \overline{\pi}(z) dz+K_2$$
i.e., there exists a constant $K_3>0$, independent from $n$, such that 
$$K_1 \sup_{x\in(2A, 2n)}\gamma^\beta(x) \leq  \left(b + 2^{\frac{1}{\beta-1}} \int_0^{+\infty} \overline{\pi}(z) dz\right) \sup_{x\in(2A, 2n)}\gamma(x) +   K_3   .$$
Finally, we deduce that 
$$K_1 \sup_{x\in(2A, 2n)}\gamma^{\beta-1}(x) \leq  \left(b + 2^{\frac{1}{\beta-1}} \int_0^{+\infty} \overline{\pi}(z) dz\right)+  \frac{ K_3}{ \sup\limits_{x\in(2A, 2n)}\gamma(x)}   $$
and letting $n\rightarrow+\infty$ yields 
$$\sup_{x\in(2A, +\infty)}\gamma^{\beta-1}(x) < +\infty.$$
This gives the upper bound since $\beta>1$.\\

%

We now look at the lower bound. Observe first from Lemmas \ref{lem:F} and \ref{lem:R} that for $x\geq0$
$$u(x)-R(x) \geq \E[1_{\{0\leq L_\e \leq x\}}\left(u(x-L_\e)- F(u(x-L_\e))\right)] \geq \Pb(0\leq L_\e< x) \left(u(x)-F(u(x))\right).$$
Using the asymptotics of $F$ and the fact that $-L$ is not a subordinator, we deduce that there exists $A>0$ large enough and $\gamma\in(0,1)$ such that
\begin{equation}\label{eq:R<gu}
\forall x\geq A,\qquad u(x)-R(x)\geq \gamma u(x).
\end{equation}
We now go back to (\ref{eq:keyR}) and write for $x\geq A$ :
$$
k\int_{x}^{+\infty} (F(u(y)) - R(y))  \Pb\left( I_{\infty} \geq x-y\right)  dy  \geq  b \gamma  u(x) + \gamma \int_A^{x} u(z) \overline{\pi}(x-z) dz.
$$
Since $u$ is decreasing and $R$ is positive on $(0,+\infty)$, we further obtain :
\begin{equation}\label{eq:startlb}
k\int_{x}^{+\infty} F(u(y))  dy  \geq   \gamma \left(b +\int_{0}^{x-A} \overline{\pi}(z)dz\right) u(x).
\end{equation}
Applying Lemma \ref{lem:F}, we then obtain that there exists a constant $K_4>0$ such that for $x\geq 2A$,
$$K_4 \geq \frac{u(x)}{\int_{x}^{+\infty} u^\beta(y)  dy }.$$
Elevating both sides to the power $\beta$ and then integrating on $(2A,z)$ with $z>2A$, we obtain
\begin{align*}
K_4^\beta (z-2A) &\geq  \left[ \frac{1}{\beta-1} \left( \int_{x}^{+\infty} u^\beta(y)   dy \right)^{1-\beta}   \right]_{2A}^z\\
&=  \frac{1}{\beta-1}\left(\int_{z}^{+\infty} u^\beta(y)   dy \right)^{1-\beta}- \frac{1}{\beta-1}\left( \int_{2A}^{+\infty} u^\beta(y)   dy \right)^{1-\beta}.
\end{align*}
As a consequence, there exists a constant $K_5\in \R$ such that for $z$ large enough
$$\left(\frac{1}{(\beta-1) K_4^\beta z+K_5}\right)^{\frac{1}{\beta-1}} \leq  \int_{z}^{+\infty} u^\beta(y)   dy.$$
%
%
Finally, for $\varepsilon\in(0,1)$, we have from the first part of the proof
$$\int_{z}^{+\infty} u^\beta(y) dy \leq u^{\varepsilon}(z) \int_{z}^{+\infty} u^{\beta-\varepsilon}(y) dy \leq  \kappa_2 (u(z))^{\varepsilon} z^{-\frac{1-\varepsilon}{\beta-1}}$$
for some constant $\kappa_2>0$, hence 
$$u(z) \geq  \kappa_2^{-\frac{1}{\varepsilon}}\left( \frac{1}{(\beta-1)K_4^\beta z + K_5}\right)^{\frac{1}{\varepsilon(\beta-1)}} z^{\frac{1-\varepsilon}{\varepsilon(\beta-1)}}$$
and 
$$\liminf_{z\rightarrow+\infty} z^{\frac{1}{\beta-1}} u(z) \geq  \left( \kappa_2 ((\beta-1) K_4^\beta)^{\frac{1}{\beta-1}}\right)^{- \frac{1}{\varepsilon} } >0 $$
which concludes the proof of the lower bound of Lemma \ref{lem:1b}.
\end{proof}

\subsubsection{The asymptotics of Theorem \ref{theo:m} when $\E[L_1]>0$}

\begin{lemma}\label{lem:mdt}
It holds 
$$ \lim_{x\rightarrow+\infty} x^{\frac{1}{\beta-1}} \int_x^{+\infty} u^\beta(y)dy = (\beta-1) \left(\frac{ \Psi^\prime(0^+)}{c_\beta \Gamma(2-\beta)}\right)^{\frac{\beta}{\beta-1}}.  $$
\end{lemma}

Point (1) of Theorem \ref{theo:m} then follows from Lemma \ref{lem:mdt} using the monotone density theorem for regularly varying functions,  since the function $x\rightarrow \int_x^{+\infty} u^\beta(y)dy $ has a monotone derivative, see for instance \cite[Theorem 1.7.2]{BGT}.

\begin{proof}
Notice first that from Lemma \ref{lem:F}, it is equivalent to show that
$$ \lim_{x\rightarrow+\infty} x^{\frac{1}{\beta-1}} \int_x^{+\infty} F(u(y))dy = \left(\frac{( \Psi^\prime(0^+))^\beta}{c_\beta \Gamma(2-\beta)}\right)^{\frac{1}{\beta-1}}.  $$
Also, from Lemma \ref{lem:1b}, there exist two constants $\kappa_1,\kappa_2>0$ such that for $x$ large enough
\begin{equation}\label{eq:1b}
\kappa_1 x^{-\frac{1}{\beta-1}} \leq \int_x^{+\infty} F(u(y))dy \leq \kappa_2 x^{-\frac{1}{\beta-1}}. 
\end{equation}
We start with the lower bound, going back to (\ref{eq:keyR}). Since $u$ is decreasing, we have
$$
k \int_{x}^{+\infty}  F(u(y)) dy \geq  \left(b  +  \int_{0}^x  \overline{\pi}(y)dy\right)  u(x)  - b R(x)   -\int_{0}^x  R(z) \overline{\pi}(x-z) dz. 
$$
i.e.
$$\frac{1}{ \left(b +\int_{0}^x \overline{\pi}(y)dy\right)}   \left(k +\frac{bR(x)+\int_0^x R(z)\overline{\pi}(x-z)dz  }{\int_{x}^{+\infty} F(u(y)) dy }\right) \geq    \frac{u(x)}{\int_{x}^{+\infty} F(u(y)) dy }. $$
From (\ref{eq:maxpi}) and (\ref{eq:asympR1}), we have the bound
\begin{align*}
\int_0^x R(z)\overline{\pi}(x-z)dz  &\leq \overline{\pi}\left(\frac{x}{2}\right) \int_0^{x/2} R(z) dz + C_R \int_{x/2}^x  z^{-\frac{\beta}{\beta-1} - \delta^\ast}  \overline{\pi}(x-z)dz \\
&\leq  2^{\frac{\beta}{\beta-1}+\delta^\ast} C_\pi  x^{-\frac{\beta}{\beta-1}-\delta^\ast} \int_0^{+\infty} R(z) dz + C_R x^{-\frac{1}{\beta-1} - \delta^\ast}   \int_{1/2}^1  z^{-\frac{\beta}{\beta-1} - \delta^\ast}  \overline{\pi}(x(1-z))dz. 
\end{align*}
Fix $\varepsilon>0$. As a consequence of the previous inequality and (\ref{eq:psi0}), we may take $A>0$ large enough such that for any $x\geq A$ :
$$\left|\frac{bR(x)+\int_0^x R(z)\overline{\pi}(x-z)dy }{\int_{x}^{+\infty} F(u(y)) dy }\right|
 \leq \varepsilon\quad \text{ and }\quad b +\int_{0}^x \overline{\pi}(y)dy\geq k\Psi^\prime(0^+)-\varepsilon.$$
This yields, from Lemma \ref{lem:F} and $x$ large enough
$$ \frac{k+\varepsilon}{k\Psi^\prime(0^+)-\varepsilon}\geq    \frac{u(x)}{\int_{x}^{+\infty} F(u(y)) dy }\geq   \frac{\beta-1}{c_\beta\Gamma(2-\beta)(1+\varepsilon)}  \frac{u(x)}{\int_{x}^{+\infty} u^\beta(y) dy }. $$
Elevating both sides to the power $\beta$ and integrating on $(A,z)$ with $z>A$, we obtain as before 
$$\left(c_\beta \frac{\Gamma(2-\beta)}{\beta-1}\frac{(k+\varepsilon)(1 +\varepsilon)}{k\Psi^\prime(0^+)-\varepsilon}\right)^\beta (z-A) \geq   \frac{1}{\beta-1}\left(\int_{z}^{+\infty} u^\beta(y) dy \right)^{1-\beta}- \frac{1}{\beta-1}\left( \int_{A}^{+\infty} u^\beta(y)   dy \right)^{1-\beta}$$
i.e.
\begin{multline*}
\int_{z}^{+\infty} u^\beta(y)   dy \geq  (\beta-1)^{-\frac{1}{\beta-1}}   \left(\left( c_\beta \frac{\Gamma(2-\beta)}{\beta-1}\frac{(k+\varepsilon)(1 +\varepsilon)}{k\Psi^\prime(0^+)-\varepsilon}\right)^\beta (z-A) \right.\\
\left.+   \frac{1}{\beta-1}\left( \int_{A}^{+\infty} u^\beta(y)   dy \right)^{1-\beta}  \right)^{-\frac{1}{\beta-1}}.
\end{multline*}
Multiplying both sides by $z^{\frac{1}{\beta-1}}$ and letting $z\rightarrow +\infty$, we deduce that
$$\liminf_{z\rightarrow+\infty} z^{\frac{1}{\beta-1}} \int_{z}^{+\infty} u^\beta(y)   dy \geq(\beta-1) \left(\frac{ k\Psi^\prime(0^+)-\varepsilon}{c_\beta \Gamma(2-\beta)(k+\varepsilon)(1+\varepsilon)}\right)^{\frac{\beta}{\beta-1}} $$
which gives the limit inferior by letting $\varepsilon\downarrow0$.\\

We now look at the upper bound. Take $\varepsilon>0$ and observe first that since $u$ is decreasing and $R$ is positive on $(0,+\infty)$,
\begin{align*}
&k\int_{x(1+\varepsilon)}^{+\infty}\left( F(u(y))-R(y)\right)  \Pb\left( I_{\infty} \geq x(1+\varepsilon)-y\right)  dy\\
 &\qquad \leq  bu(x)  + \int_0^{x} u(z) \overline{\pi}(x(1+\varepsilon)-z) dz + \int_{x}^{x(1+\varepsilon)} u(z) \overline{\pi}(x(1+\varepsilon)-z) dz - \int_{-\infty}^0 R(z)\overline{\pi}(x(1+\varepsilon)-z) dz \\
&\qquad \leq  \left(b +  \int_{x}^{x(1+\varepsilon)} \overline{\pi}(x(1+\varepsilon)-z)dz\right)  u(x)  +  \overline{\pi}(x\varepsilon) \int_0^{x} u(z)  dz + \int_{-\infty}^0 |R(z)|\overline{\pi}(x(1+\varepsilon)-z) dz\\
&\qquad \leq k \Psi^\prime(0^+) u(x)  +  \overline{\pi}(x\varepsilon) \E[{\bf M}] + \overline{\pi}(x)  \int_{-\infty}^0 |R(z)| dz.
\end{align*}
We now rewrite this expression under the form 
\begin{equation}\label{eq:FM}
k\int_{x}^{+\infty}F(u(y)) dy
\leq  k \Psi^\prime(0^+) u(x)  +  \overline{\pi}(x\varepsilon) \E[{\bf M}] +\overline{\pi}(x)  \int_{-\infty}^0 |R(z)| dz+\Delta_R(x) + \Delta_F(x)
\end{equation}
where
$$\Delta_R(x) := k \int_{x(1+\varepsilon)}^{+\infty}R(y) \Pb\left( I_{\infty} \geq x(1+\varepsilon)-y\right) dy$$
and
$$\Delta_F(x):=k\int_{x}^{x(1+\varepsilon)}F(u(y)) dy+ k\int_{x(1+\varepsilon)}^{+\infty}F(u(y))  \Pb\left( I_{\infty} < x(1+\varepsilon)-y\right)  dy  . $$
We now proceed as for the lower bound and start by controlling the remainders thanks to (\ref{eq:1b}). From (\ref{eq:asympR1}), we have, using a change of variables, 
$$
\left|\frac{\Delta_R(x)}{k\int_{x}^{+\infty}F(u(y)) dy}  \right| \leq  \frac{1}{\kappa_1}C_Rx^{-\delta^\ast} \int_{1+\varepsilon}^{+\infty}z^{-\frac{\beta}{\beta-1}-\delta^\ast}  dz  \xrightarrow[x\rightarrow+\infty]{} 0
$$
while, using Lemmas \ref{lem:F} and \ref{lem:1b}, as well as the monotone convergence theorem, 
\begin{align*}
\left|\frac{\Delta_F(x)}{k\int_{x}^{+\infty}F(u(y)) dy}  \right| &\leq  \frac{\kappa_2}{\kappa_1} x^{\frac{1}{\beta-1}} \left( \int_{x}^{x(1+\varepsilon)} y^{-\frac{\beta}{\beta-1}}dy + 
\int_{x(1+\varepsilon)}^{+\infty}y^{-\frac{\beta}{\beta-1}} \Pb\left( I_{\infty} < x(1+\varepsilon)-y\right)  dy\right)\\
&\leq  \frac{\kappa_2}{\kappa_1}  \left(\varepsilon   +  \int_{1+\varepsilon}^{+\infty}z^{-\frac{\beta}{\beta-1}} \Pb\left( I_{\infty} < x(1+\varepsilon)-xz\right)  dz\right)\\
&\xrightarrow[x\rightarrow+\infty]{}   \frac{\kappa_2}{\kappa_1}\varepsilon.
\end{align*}
From (\ref{eq:maxpi}), the last terms are also negligible :
\begin{align*}
\left|\frac{ \overline{\pi}(x\varepsilon)\E[{\bf M}] + \overline{\pi}(x)  \int_{-\infty}^0 |R(z)| dz  }{\int_{x}^{+\infty}F(u(y)) dy}  \right| \leq    \frac{C_\pi}{\kappa_1} x^{-1-\delta^\ast}  \left(\E[{\bf M}]  \varepsilon^{-\frac{\beta}{\beta-1}-\delta^\ast} +   \int_{-\infty}^0 |R(z)| dz \right) \xrightarrow[x\rightarrow+\infty]{} 0.
\end{align*}
As a consequence,
$$\limsup_{x\rightarrow+\infty}G_\varepsilon(x) :=  \limsup_{x\rightarrow+\infty}\frac{ \overline{\pi}(x\varepsilon) \E[{\bf M}] +\overline{\pi}(x)  \int_{-\infty}^0 |R(z)| dz+\Delta_R(x) + \Delta_F(x)}{k\int_{x}^{+\infty}F(u(y)) dy }\leq \frac{\kappa_2}{\kappa_1}\varepsilon.$$
Finally, taking $A>0$ large enough, we may rewrite Equation (\ref{eq:FM}) for $x\geq A$ under the form :
$$\frac{1 - \sup_{r\geq A}G_\varepsilon(r)}{\Psi^\prime(0^+)}  \leq \frac{u(x)}{\int_{x}^{+\infty} F(u(y)) dy } \leq  \frac{\beta-1}{c_\beta\Gamma(2-\beta)(1-\varepsilon)}  \frac{u(x)}{\int_{x}^{+\infty} u^\beta(y) dy }. $$
Elevating to the power $\beta$ and integrating on $(A,z)$ with $z>A$, we deduce that
$$ (\beta-1)\left(\frac{c_\beta\Gamma(2-\beta)(1-\varepsilon)  (1 - \sup_{r\geq A}G_\varepsilon(r))}{(\beta-1)\Psi^\prime(0^+)}  \right)^\beta (z-A) \leq \left(\int_{z}^{+\infty} u^\beta(y) dy \right)^{1-\beta}-\left( \int_{A}^{+\infty} u^\beta(y)   dy \right)^{1-\beta}$$
and proceeding as before, we obtain that
$$\limsup_{z\rightarrow+\infty} z^{\frac{1}{\beta-1}} \int_{z}^{+\infty} u^\beta  (y) dy \leq(\beta-1)\left(\frac{\Psi^\prime(0^+)} {c_\beta\Gamma(2-\beta)(1-\varepsilon)  (1 - \sup_{r\geq A}G_\varepsilon(r))} \right)^{\frac{\beta}{\beta-1}}. $$
The upper bound follows by letting $A\uparrow +\infty$ and $\varepsilon\downarrow0$.
\end{proof}

\subsection{The case $\E[L_1]<0$.}
The situation where $\E[L_1]<0$ is  easier to deal with as the assumption that $L$ admits some (positive) exponential moments will allow us to work with Laplace transforms. We first check that $u$ is  indeed at least exponentially decreasing.

\begin{lemma}\label{lem:Fu}
It holds
$$\int_0^{+\infty} e^{\omega z} F(u(z))dz <+\infty.$$
\end{lemma} 
\begin{proof}
Notice first that $\Psi$ being convex, we have $\Psi(\lambda)\leq 0$ for all $\lambda \in [0,\omega]$. In particular, $\E[e^{\lambda L_\e}]  = 1/(1-\Psi(\lambda))<+\infty$ and we deduce from the Wiener Hopf factorisation that $\E[e^{\lambda S_\e}]$ is also finite for  $\lambda \in [0,\omega]$. Using Lemma \ref{lem:R}, this implies that 
$$\int_0^{+\infty} e^{\omega x} R(x) dx \leq \frac{1}{\omega} \E\left[e^{\omega S_\e}\right]<+\infty.$$
%
%
%
Now, to prove Lemma \ref{lem:Fu}, we start by integrating Equation (\ref{eq:usansS}) against  $\exp\left(\omega x - \frac{1}{n}e^{\omega x}\right)$ on $(0,+\infty)$, where $n>0$. This yields, after a change of variables
\begin{multline*}
\E\left[\int_0^{+\infty}  1_{\{z\geq -L_\e\}}e^{\omega L_\e} e^{\omega z - \frac{1}{n}e^{\omega z+\omega L_\e}} F(u(z) )dz \right]\\ 
=\E\left[\int_0^{+\infty}  1_{\{z\geq -L_\e\}}e^{\omega L_\e} e^{\omega z - \frac{1}{n}e^{\omega z+\omega L_\e}} u(z) dz \right]
-\int_0^{+\infty} e^{\omega x - \frac{1}{n}e^{\omega x}} (u(x)-R(x)) dx 
\end{multline*}
i.e., since $F$ is positive,
\begin{multline*}
\E\left[\int_0^{+\infty}  1_{\{L_\e\geq0\}}e^{\omega L_\e} e^{\omega z - \frac{1}{n}e^{\omega z+\omega L_\e}} F(u(z) )dz \right]\\ 
\leq \E\left[\int_0^{+\infty}  e^{\omega L_\e} e^{\omega z - \frac{1}{n}e^{\omega z+\omega L_\e}} u(z) dz \right]
-\int_0^{+\infty} e^{\omega x - \frac{1}{n}e^{\omega x}} (u(x)-R(x)) dx. 
\end{multline*}
Integrating by parts the terms in $u$ on the right-hand side, we obtain 
$$ \frac{n}{\omega} \left[ \left(e^{ - \frac{1}{n}e^{\omega z}}  - e^{ - \frac{1}{n}e^{\omega z+\omega L_\e}}\right) u(z)\right]_0^{+\infty } + \frac{n}{\omega} \E\left[\int_0^{+\infty}\left( e^{- \frac{1}{n}e^{\omega z+\omega L_\e}} - e^{- \frac{1}{n}e^{\omega z}}\right) u^\prime(z)   dz \right]. $$
Since $u^\prime$ is negative, we deduce from Jensen inequality and the definition of $\omega$ that this last expression is smaller than
$$ \frac{n}{\omega}  \E\left[e^{ - \frac{1}{n}e^{\omega L_\e}}-e^{ - \frac{1}{n}}\right]  +  \int_0^{+\infty}\left( e^{- \frac{1}{n}\E\left[e^{\omega z+\omega L_\e}\right]} - e^{- \frac{1}{n}e^{\omega z}}\right) u^\prime(z)   dz = \frac{n}{\omega} \E\left[ e^{ - \frac{1}{n}e^{\omega L_\e}}-e^{ - \frac{1}{n}}\right]\xrightarrow[n\rightarrow+\infty]{}0.  $$
As a consequence, we obtain the upper bound
$$\limsup_{n\rightarrow +\infty}  \E\left[\int_0^{+\infty}  1_{\{L_\e \geq 0\}}e^{\omega L_\e} e^{\omega z - \frac{1}{n}e^{\omega z+\omega L_\e}} F(u(z) )dz\right] \leq  \int_0^{+\infty} e^{\omega x }R(x) dx <+\infty. $$
Letting $n\rightarrow +\infty$ and applying the monotone convergence theorem yields
$$  \E\left[ 1_{\{L_\e \geq 0\}}e^{\omega L_\e} \right]  \int_0^{+\infty} e^{\omega z} F(u(z) )dz <+\infty$$
which proves from Lemma \ref{lem:F} that $u$ is at least exponentially decreasing.
\end{proof}

Taking the two-sided Laplace transform of Equation (\ref{eq:usansS}) for $\lambda$ small enough, we obtain :
$$\int_\R e^{\lambda x} u(x) dx =  \E\left[e^{\lambda L_\e}\right]\int_\R e^{\lambda x} u(x) dx  -  \E\left[e^{\lambda L_\e}\right] \int_\R e^{\lambda x} F(u(x)) dx -  \int_\R e^{\lambda x} R(x) dx$$
i.e., from the definition of $\Psi$ as the Laplace exponent of $L$, 
\begin{equation}\label{eq:Lapu}
\int_\R e^{\lambda x} (F(u(x))-R(x)) dx = \Psi(\lambda)\int_\R e^{\lambda x} (u(x)-R(x)) dx.
\end{equation}
Furthermore, for $0<\lambda <\omega$,
$$-\frac{1}{\Psi(\lambda)} = \int_0^{+\infty} e^{t \Psi(\lambda)} dt = \int_0^{+\infty} \int_0^{+\infty} e^{\lambda z} \Pb(L_t\in dz) dt = \int_0^{+\infty} e^{\lambda z} U(dz)$$
where $U$ denotes the potential of $L$. We then set, following \cite{BeDo},
$$U(dz) = e^{-\omega z} U^\ast(dz)$$ where $U^\ast$ denotes the potential of the associated L\'evy process $L^\ast$ whose Laplace exponent is given by $\Psi^\ast(\lambda) =  \Psi(\lambda+\omega)$. In particular, $(\Psi^\ast)^\prime(0) = \Psi^\prime(\omega^+) =\E[L_1^\ast] >0$ by the convexity of $\Psi$. Note that $\E[L_1^\ast]$ is necessarily finite since we have assumed that $\Psi$ is analytic in a neighborhood of $\omega$. Inverting Formula (\ref{eq:Lapu}), we obtain :
\begin{align}
\label{eq:u-RU} u(x)-R(x)  &=  - \int_\R \left(F(u(x-y))-R(x-y)\right) e^{-\omega y} U^\ast(dy)\\
\notag & = - e^{-\omega x}\int_0^{+\infty}  \E_{-x}\left[\left(F(u(-L^\ast_t)) - R(-L_t^\ast)\right) e^{-\omega L_t^\ast}\right] dt.
 \end{align}
Applying the renewal theorem \cite[Chapter I, Theorem 21]{Ber}, we deduce that 
$$ -\int_0^{+\infty}  \E_{-x}\left[\left(F(u(-L^\ast_t)) - R(-L_t^\ast)\right) e^{-\omega L_t^\ast}\right] dt
\xrightarrow[x\rightarrow +\infty]{}  -\frac{1}{\E[L_1^\ast]}\int_\R e^{\omega z}\left(F(u(z)) -R(z) \right)  dz$$
which is finite from Lemma \ref{lem:Fu}. It remains to check that this constant is not null. We shall proceed by contradiction. Let us assume that $\int_\R e^{\omega z}\left(F(u(z)) -R(z) \right)  dz=0$. Dividing (\ref{eq:Lapu}) by $\lambda-\omega$ and letting $\lambda \uparrow \omega$, we obtain
$$\int_\R e^{\omega x} x (F(u(x))-R(x)) dx = \Psi^\prime(\omega^-)\int_\R e^{\omega x} (u(x)-R(x)) dx $$
which implies that $\int_\R e^{\omega x} (u(x)-R(x)) dx$ is finite. As a consequence, we deduce by analytic continuation that the equality
$$\frac{1}{\Psi(\lambda)} \int_\R e^{\lambda x}\left(F(u(x)) -R(x) \right)  dx = \int_\R e^{\lambda x} (u(x)-R(x)) dx$$
also holds for $\lambda \in (\omega, \omega +2\varepsilon)$ with $\varepsilon>0$ small enough. In particular, this implies that 
$$ \int_\R e^{(\omega +\varepsilon)  x} u(x) dx = \E[e^{(\omega+\varepsilon) {\bf M}}] <+\infty.$$  
But, looking only at one path of the branching process $X$, we have 
$\E[e^{(\omega+\varepsilon) {\bf M}}] \geq \E[e^{(\omega+\varepsilon) S_\zeta}]$ where $\zeta$ denotes the extinction time of  $X$, which is independent of $S$.
Since from Lemma \ref{lem:F} the generating function of the offspring distribution satisfies 
$$\sum_{k\geq0} s^k p_k = s + F(1-s) =  s + (1-s)^{\beta} \times L(1-s)$$
where $L$ is a slowly varying function, we deduce from \cite[Theorem 2]{Bor} that the asymptotics of the tail of $\zeta$ is given by 
$\Pb(\zeta >t) \equi_{t\rightarrow +\infty} t^{-\frac{1}{\beta-1}}L^\ast(t)$
for some  slowly varying function $L^\ast$. As a consequence,
$$\E[e^{(\omega+\varepsilon) S_\zeta}] \geq \E[e^{(\omega+\varepsilon) L_\zeta}] = \int_\R e^{t \Psi(\omega+\varepsilon)} \Pb(\zeta \in dt) =+\infty$$
since $\Psi(\omega+\varepsilon)>0$. This contradicts the finiteness of the $\omega+\varepsilon$ exponential moment of ${\bf M}$.

\qed

\section{Proof of Theorem \ref{theo:2}}

We now tackle the case when $L$ no longer admits moments of order 2. In the following, we need to separate the two cases $\alpha\in(0,1]$ and $\alpha\in (1,2)$, as in the latter case, we will have to deal with an extra term since the expectation of $L_1$ is finite.

\subsection{The case $\alpha\in(0,1]$}
To simplify the notation, we set
$$\eta_\alpha(\lambda) = \begin{cases}\Gamma(1-\alpha)\ell_\alpha(1/\lambda) \lambda^{\alpha-1} &\qquad \text{if }\alpha\in(0,1)\\
- \ell_1 \ln(\lambda) &\qquad   \text{if }\alpha=1
\end{cases} $$
so that  from Assumption \ref{assum} and the Tauberian theorem (\ref{eq:taub}), together with a direct calculation when $\alpha=1$, we have 
\begin{equation}\label{eq:SeLe1}
\frac{1-\E\left[e^{-\lambda S_\e}\right]}{\lambda} \,\equi_{\lambda\downarrow 0}\, \frac{1-\E\left[e^{-\lambda L_\e^+}\right]}{\lambda} \,\equi_{\lambda\downarrow 0}\, \eta_\alpha(\lambda).
\end{equation}

\subsubsection{A key lemma}

The proof will rely on the following Lemma which will be used repeatedly in the sequel :
\begin{lemma}\label{lem:key1}
Let $f$ be a positive and non-increasing function such that $\lim\limits_{x\rightarrow+\infty}f(x)=0$. Then
$$\lim_{\lambda\rightarrow0} \frac{1}{\eta_\alpha(\lambda)} \left(\L[f](\lambda) - \int_0^{+\infty}e^{-\lambda x}\E\left[1_{\{S_\e< x\}}f(x-L_\e)\right] dx \right) =0.$$
\end{lemma}

\begin{proof}
Observe first that since $L_\e \leq S_\e$ a.s. and $f$ is non-increasing, we have  $f(x-L_\e) \leq f(x-S_\e)$ a.s. Applying the Fubini-Tonelli theorem to compute the convolution product,  this implies that : 
 $$ \int_0^{+\infty}e^{-\lambda x}\E\left[1_{\{S_\e< x\}}f(x-L_\e)\right] dx \leq \E\left[e^{-\lambda S_\e}\right]  \L[f](\lambda).$$
As a consequence, for all $\lambda>0$,
\begin{equation}
\label{eq:positif}\L[f](\lambda) - \int_0^{+\infty}e^{-\lambda x}\E\left[1_{\{S_\e< x\}}f(x-L_\e)\right] dx  
 \geq \left(1-  \E\left[e^{-\lambda S_\e}\right] \right) \L[f](\lambda) \geq0.
\end{equation}
Conversely,  using that $1_{\{S_\e< x\}} - 1_{\{L_\e< x\}}\leq 0 $ a.s., we have 
\begin{align*}
&\int_0^{+\infty} e^{-\lambda x}\E\left[1_{\{S_\e< x\}}f(x-L_\e)\right] dx \\
&\qquad =\int_0^{+\infty} e^{-\lambda x}\E\left[\left(1_{\{S_\e< x\}} - 1_{\{L_\e< x\}}\right)f(x-L_\e)\right] dx + \int_0^{+\infty} e^{-\lambda x}\E\left[1_{\{L_\e< x\}}f(x-L_\e)\right] dx\\
&\qquad \geq  f(0) \frac{\E\left[e^{-\lambda S_\e}\right] - \E\left[e^{-\lambda L_\e^+}\right]}{\lambda} + \E\left[e^{-\lambda L_\e}1_{\{L_\e\geq0\}}\right] \L[f](\lambda)+  \int_0^{+\infty} e^{-\lambda x}\E\left[1_{\{L_\e< 0\}}f(x-L_\e)\right] dx. 
\end{align*}
Then, the Fubini-Tonelli theorem and a change of variable in the last integral yields :
\begin{align*} \int_0^{+\infty}e^{-\lambda x}  \E\left[1_{\{L_\e<0\}}f(x-L_\e)\right] dx &=  \E\left[1_{\{L_\e<0\}} \int_{-L_\e}^{+\infty} e^{-\lambda z - \lambda L_\e } f(z)dz\right] \\
 &\geq  \Pb\left(L_\e<0\right)\mathcal{L}[f](\lambda) -  \E\left[1_{\{L_\e<0\}} \int_{0}^{-L_\e} e^{-\lambda z  } f(z)dz\right].
\end{align*}
Plugging everything together, we thus obtain
\begin{multline}\label{eq:boundL1}
\left(\L[f](\lambda) - \int_0^{+\infty}e^{-\lambda x}\E\left[1_{\{S_\e< x\}}f(x-L_\e)\right] dx \right)\\
\qquad \qquad \leq   f(0)\frac{ \E\left[e^{-\lambda L_\e^+}\right]-\E\left[e^{-\lambda S_\e}\right] }{\lambda} + \left(1-\E\left[e^{-\lambda L_\e^+}\right]\right)\mathcal{L}[f](\lambda)+ \E\left[1_{\{L_\e<0\}} \int_{0}^{-L_\e} e^{-\lambda z  } f(z)dz\right].
\end{multline}
Now, from (\ref{eq:SeLe1}), the first term on the right-hand side converges towards 0,
$$\lim_{\lambda\rightarrow0} \frac{1}{\eta_\alpha(\lambda)} \frac{ \E\left[e^{-\lambda L_\e^+}\right]-\E\left[e^{-\lambda S_\e}\right] }{\lambda} =0$$
and as above, using a change of variable and the monotone convergence theorem,
$$\lim_{\lambda\rightarrow0} \frac{1}{\eta_\alpha(\lambda)}\left(1-\E\left[e^{-\lambda L_\e^+}\right]\right)\mathcal{L}[f]
=\lim_{\lambda\rightarrow0} \frac{1}{\eta_\alpha(\lambda)}  \frac{1-\E\left[e^{-\lambda L_\e^+}\right]}{\lambda} \int_0^{+\infty} e^{-z} f\left(\frac{z}{\lambda}\right) dz=0.$$
Finally, to show that the last term on the right-hand side of (\ref{eq:boundL1}) also converges towards 0, let us take $\varepsilon>0$. By assumption on $f$, there exists $A_\varepsilon$ such that $f(x)\leq \varepsilon$ for any $x\geq A_\varepsilon$. Then 
\begin{align*}
&  \E\left[1_{\{L_\e<0\}} \int_{0}^{-L_\e} e^{-\lambda z  } f(z)dz\right]\\
& \qquad \leq  \E\left[\int_0^{A_\varepsilon} e^{-\lambda z  } f(z)1_{\{L_\e<0\}} 1_{\{z<-L_\e\}}dz + \varepsilon \int_{A_\varepsilon}^{+\infty} e^{-\lambda z  } 1_{\{L_\e<0\}} 1_{\{z<-L_\e\}}dz \right]\\
&\qquad  \leq A_\varepsilon f(0) + \varepsilon  \int_{A_\varepsilon}^{+\infty} e^{-\lambda z  } \Pb(L_\e<-z)dz\\
&\qquad \leq A_\varepsilon f(0) + \varepsilon  C_\alpha  \int_{0}^{+\infty} e^{-\lambda z  } \Pb(L_\e>z)dz
\end{align*}
where the last inequality follows from Assumption (\ref{hyp1bis}), by taking $A_\varepsilon$ large enough. As a consequence, we deduce that 
$$ \limsup_{\lambda \rightarrow 0}\frac{1}{\eta_\alpha(\lambda)}\E\left[1_{\{L_\e<0\}} \int_{0}^{-L_\e} e^{-\lambda z  } u(z)dz\right] \leq \varepsilon C_\alpha$$
which proves Lemma \ref{lem:key1}.
\end{proof}

\subsubsection{Proof of Theorem \ref{theo:2} when $\alpha\in(0,1)$}

Let us take the Laplace transform of the integral equation (\ref{eq:keyu}) satisfied by $u$. We obtain
\begin{multline}\label{eq:Lap1}
\L[u](\lambda) - \int_0^{+\infty} e^{-\lambda x}  \E\left[1_{\{S_\e<x\}}u(x-L_\e)\right]dx\\ = \frac{1-\E\left[e^{-\lambda S_\e}\right]}{\lambda} - \int_0^{+\infty} e^{-\lambda x} \E\left[1_{\{S_\e<x\}}F(u(x-L_\e))\right]dx.
\end{multline}
Applying Lemma \ref{lem:key1} with $f=u$, we deduce thanks to (\ref{eq:SeLe1}) that
$$\int_0^{+\infty}  e^{-\lambda x}  \E\left[1_{\{S_\e< x\}}F(u(x-L_\e))\right] dx \equi_{\lambda\rightarrow0} \eta_\alpha(\lambda). $$
Applying next Lemma \ref{lem:key1} with $f=F\circ u$, we obtain 
$$\int_0^{+\infty}  e^{-\lambda x} F(u(x)) dx \equi_{\lambda\rightarrow0} \eta_\alpha(\lambda). $$
As a consequence, when $\alpha\in(0,1)$, we conclude from the Tauberian theorem (\ref{eq:taub}), since $F\circ u$ is decreasing, that 
$$ F(u(x)) \equi_{x\rightarrow +\infty} \ell_\alpha(x) x^{-\alpha}.$$
Finally, from Lemma \ref{lem:F}, since $u(x)\xrightarrow[x\rightarrow+\infty]{}0$, this implies that 
$$u(x)  \equi_{x\rightarrow +\infty} \left(\frac{\beta-1}{c_\beta \Gamma(2-\beta)}\right)^{\frac{1}{\beta}}  \ell_\alpha^{\frac{1}{\beta}}(x)  x^{-\frac{\alpha}{\beta}}$$
which is the announced asymptotics. When $\alpha=1$, we obtain, using the integrated version of Karamata's Tauberian theorem \cite[Theorem 1.7.1]{BGT}
\begin{equation}\label{eq:Fuln}
\int_0^{x} F(u(z))dz  \equi_{x\rightarrow +\infty}  \ell_1 \ln(x)
\end{equation}
but we unfortunately cannot differentiate this equivalence as such. We shall end the proof of this case after dealing with the situation $\alpha\in(1,2)$.

\qed
%
%

\subsection{The case $\alpha\in(1,2)$}
The main difference with the previous case is that the expectations of $L_\e$ and $S_\e$ are now finite. As a consequence, Assumption \ref{assum} and the Tauberian theorem (\ref{eq:taub}) yields the asymptotics 
\begin{equation}\label{eq:xP}
 \int_0^{+\infty}e^{-\lambda x}x\Pb\left(S_\e\geq x\right)dx=    \frac{1-\E\left[e^{-\lambda S_\e}\right] - \lambda \E\left[S_\e e^{-\lambda S_\e}\right]}{\lambda^2}\,\equi_{\lambda\downarrow0}\, \Gamma(2-\alpha) \ell_\alpha\left(\frac{1}{\lambda}\right) \lambda^{\alpha-2}  
 \end{equation}
and likewise for $L_\e$.
We proceed as before and start by a result similar to Lemma \ref{lem:key1}.

\subsubsection{A key lemma}

\begin{lemma}\label{lem:key2}
Let $f$ be a positive, differentiable and non-increasing function such that
\begin{equation}\label{eq:assum2}
 \lim_{\lambda\rightarrow0} \frac{1}{\ell_\alpha(1/\lambda)} \lambda^{2-\alpha}   \mathcal{L}[f](\lambda)=0.
 \end{equation}
We write $xf$ for the function $x\rightarrow xf(x)$.
Then
$$\lim_{\lambda\rightarrow0} \frac{1}{\ell_\alpha(1/\lambda)} \lambda^{2-\alpha}\left(\L[xf](\lambda) - \int_0^{+\infty}e^{-\lambda x}x\E\left[1_{\{S_\e< x\}}f(x-L_\e)\right] dx \right) =0.$$
\end{lemma}

\begin{proof}
Observe first that using the monotony of $f$ and the decomposition $x=x-S_\e+S_\e$, we have
\begin{equation}\label{eq:maxxf}
 \int_0^{+\infty}  e^{-\lambda x} x \E\left[1_{\{S_\e<x\}}f(x-L_\e)\right] dx \leq \E\left[e^{-\lambda S_\e}\right] \mathcal{L}[xf](\lambda) +  \E\left[S_\e e^{-\lambda S_\e}\right] \mathcal{L}[f](\lambda). 
 \end{equation}
This yields the lower bound
\begin{multline}\label{eq:lo}
\liminf_{\lambda\rightarrow0} \frac{1}{\ell_\alpha(1/\lambda)} \lambda^{2-\alpha}\left(\L[xf](\lambda) - \int_0^{+\infty}e^{-\lambda x}x\E\left[1_{\{S_\e< x\}}f(x-L_\e)\right] dx\right) \\
\qquad \geq - \liminf_{\lambda\rightarrow0} \frac{1}{\ell_\alpha(1/\lambda)} \lambda^{2-\alpha}   \E\left[S_\e e^{-\lambda S_\e}\right] \mathcal{L}[f](\lambda)=0 
\end{multline}
since $\E[S_\e]<+\infty$. On the other hand, we have 
\begin{multline*}
 \int_0^{+\infty}  e^{-\lambda x}  x\E\left[1_{\{S_\e<x\}}f(x-L_\e)\right]dx\\
 \geq   f(0) \int_0^{+\infty}e^{-\lambda x}x\left( \Pb\left(L_\e\geq x\right)- \Pb\left(S_\e\geq x\right)\right) dx + \int_0^{+\infty}e^{-\lambda x}x  \E\left[1_{\{L_\e<x\}}f(x-L_\e)\right]dx 
\end{multline*}
and we need to study the last term. First, we write the bound, using again the decomposition $x=x-L_\e+L_\e$ :
\begin{multline*}
\int_0^{+\infty}e^{-\lambda x}x  \E\left[1_{\{L_\e<x\}}f(x-L_\e)\right]dx 
\\\geq \E\left[e^{-\lambda L_\e}1_{\{L_\e\geq 0\}}\right]\mathcal{L}[xf](\lambda) + \E\left[L_\e e^{-\lambda L_\e}1_{\{L_\e\geq0\}}\right]\mathcal{L}[f](\lambda)+\int_0^{+\infty}e^{-\lambda x}x  \E\left[1_{\{L_\e<0\}}f(x-L_\e)\right]dx.\end{multline*}
Then, after a change of variable, the last term is seen to be greater than
\begin{align*}
&\int_0^{+\infty}e^{-\lambda x}x  \E\left[1_{\{L_\e<0\}}f(x-L_\e)\right]dx \\
&\qquad=\E\left[ 1_{\{L_\e<0\}} \int_{-L_\e}^{+\infty}e^{-\lambda (z+L_\e)} (z+L_\e)  f(z)dz\right]\\
&\qquad \geq \E\left[ 1_{\{L_\e<0\}} \int_{0}^{+\infty}e^{-\lambda z} (z+L_\e)  f(z)dz\right] - \E\left[ 1_{\{L_\e<0\}} \int_{0}^{-L_\e}e^{-\lambda z} (z+L_\e)  f(z)dz\right].
\end{align*}
Plugging everything together, we arrive at 
\begin{multline}\label{eq:up}
\mathcal{L}[xf](\lambda) -\int_0^{+\infty}e^{-\lambda x}x  \E\left[1_{\{L_\e<x\}}f(x-L_\e)\right]dx 
\\\leq f(0) \int_0^{+\infty}e^{-\lambda x}x\left( \Pb\left(S_\e\geq x\right)- \Pb\left(L_\e\geq x\right)\right) dx +  \left(1-\E\left[e^{-\lambda L_\e^+}\right]\right)\mathcal{L}[xf](\lambda) \\- \E\left[L_\e e^{-\lambda L_\e^+}\right]\mathcal{L}[f](\lambda)+ \E\left[1_{\{L_\e<0\}}\int_0^{-L_\e} e^{-\lambda z}zf(z)dz\right] .
\end{multline}
Multiplying both sides by $\lambda^{2-\alpha}/\ell_\alpha(1/\lambda)$ and letting $\lambda \downarrow0$, we deduce that the limits of the first and third terms on the right-hand side are null thanks to (\ref{eq:xP}), (\ref{eq:assum2}) and  the fact that $\E[|L_\e|]<+\infty$. For the second term, observe that since $f$ is non-increasing, integrating by parts the second Laplace transform,
$$\L[f](\lambda) - \lambda\L[xf](\lambda) = -\int_0^{+\infty} e^{-\lambda x} x f^\prime(x)dx \geq0$$
hence, still from (\ref{eq:assum2}),
$$\limsup_{\lambda\rightarrow 0} \frac{1}{\ell_\alpha(1/\lambda)} \lambda^{3-\alpha}\L[xf](\lambda) \leq \limsup_{\lambda\rightarrow 0} \frac{1}{\ell_\alpha(1/\lambda)} \lambda^{2-\alpha}\L[f](\lambda)= 0$$
which is also null.
To compute the limit of the last term, let us take as before $\varepsilon>0$. 
\begin{align*}
&\frac{ \lambda^{2-\alpha}}{\ell_\alpha(1/\lambda)} \E\left[1_{\{L_\e<0\}} \int_{0}^{-L_\e} e^{-\lambda z  } zf(z)dz\right] \\
&\qquad \leq  \frac{\lambda^{2-\alpha}}{\ell_\alpha(1/\lambda)}  \E\left[1_{\{L_\e<0\}}  \int_{0}^{A_\varepsilon} 1_{\{z< -L_\e\}}e^{-\lambda z  } zf(z)dz\right] + \varepsilon \frac{\lambda^{2-\alpha}}{\ell_\alpha(1/\lambda)}  \E\left[1_{\{L_\e<0\}} \int_{A_\varepsilon}^{+\infty}1_{\{z< -L_\e\}}  e^{-\lambda z  } z dz\right]\\
&\qquad \leq 
\frac{ \lambda^{2-\alpha}}{\ell_\alpha(1/\lambda)}   A_\varepsilon^2\,f(0)  + \varepsilon   \frac{ \lambda^{2-\alpha}}{\ell_\alpha(1/\lambda)}  \int_{A_\varepsilon}^{+\infty} e^{-\lambda z  } z\Pb(L_\e<-z)dz\\
&\qquad \leq \frac{ \lambda^{2-\alpha}}{\ell_\alpha(1/\lambda)}   A_\varepsilon^2\,f(0)  + \varepsilon C_\alpha  \frac{ \lambda^{2-\alpha}}{\ell_\alpha(1/\lambda)}  \int_{0}^{+\infty} e^{-\lambda z  } z\Pb(L_\e>z)dz \xrightarrow[\lambda\downarrow0]{} \varepsilon C_\alpha 
\end{align*}
where we used Assumption \ref{assum} in the last inequality, and the asymptotics (\ref{eq:xP}) to compute the limits.
This concludes the proof of Lemma \ref{lem:key2}.
\end{proof}

\subsubsection{Proof of Theorem \ref{theo:2} when $\alpha\in(1,2)$}

Multiplying Equation (\ref{eq:keyu}) by $x$ and taking the Laplace transform, we deduce that :
\begin{multline*}
\L[xu](\lambda) - \int_0^{+\infty} e^{-\lambda x}  x\E\left[1_{\{S_\e< x\}}u(x-L_\e)\right] dx\\ = \int_0^{+\infty} e^{-\lambda x} x  \Pb(S_\e\geq x) dx - \int_0^{+\infty} e^{-\lambda x} x \E\left[1_{\{S_\e< x\}}F(u(x-L_\e))\right] dx.
\end{multline*}
Assume for the time being that $u$ and $F\circ u$ satisfy assumption (\ref{eq:assum2}). Applying Lemma \ref{lem:key2} with $f=u$, we deduce that
$$\int_0^{+\infty} e^{-\lambda x} x \E\left[1_{\{S_\e< x\}}F(u(x-L_\e))\right] dx \equi_{\lambda\rightarrow0} \Gamma(2-\alpha) \lambda^{\alpha-2} \ell_\alpha(1/\lambda)$$
and, then, applying again Lemma \ref{lem:key2} this time with $F\circ u$,
$$\int_0^{+\infty} e^{-\lambda x} x F(u(x))dx \equi_{\lambda\rightarrow0} \Gamma(2-\alpha) \lambda^{\alpha-2} \ell_\alpha(1/\lambda).$$
Finally, integrating by parts and applying the Tauberian theorem (\ref{eq:taub}), we thus deduce that
$$\int_0^{x} z F(u(z))dz  \equi_{x\rightarrow+\infty} \frac{1}{2-\alpha}x^{2-\alpha} \ell_\alpha(x)$$
and Theorem \ref{theo:2} follows from the monotone density theorem and Lemma \ref{lem:F}. \\

\noindent
It remains thus to check that 
$$ \lim_{\lambda\rightarrow0} \frac{\lambda^{2-\alpha}}{\ell_\alpha(1/\lambda)} \mathcal{L}[F\circ u](\lambda)=0  \qquad \text{ and }\qquad  \lim_{\lambda\rightarrow0} \frac{\lambda^{2-\alpha}}{\ell_\alpha(1/\lambda)}    \mathcal{L}[u](\lambda)=0.$$
The first limit is a consequence of the fact that $F\circ u$ is integrable. Indeed, starting from (\ref{eq:Lap1}), in which both sides are positive thanks to (\ref{eq:positif}), and recalling that $F\circ u$ is decreasing, we have
\begin{align*}
  \frac{1-\E\left[e^{-\lambda S_\e}\right]}{\lambda} &\geq \int_0^{+\infty} e^{-\lambda x} \E\left[1_{\{S_\e<x\}}F(u(x-L_\e))\right]dx \\
  &\geq  \E\left[1_{\{L_\e\geq0\}}  \int_{S_\e}^{+\infty} e^{-\lambda x} F(u(x)) dx\right] \\
  &\geq \Pb\left( L_\e\geq0,\, S_\e< 1\right)\int_1^{+\infty} e^{-\lambda x}   F(u(x)) dx 
\end{align*}
which implies from the monotone convergence theorem that
$$ \Pb\left( L_\e\geq0,\, S_\e< 1\right) \, \int_1^{+\infty} F(u(x)) dx \leq \E\left[S_\e\right]<\infty.$$
The second limit, for $u$, is more involved. Fix $\varepsilon>0$ small enough such that $  \frac{1-\varepsilon}{\alpha} + \frac{1}{\beta} >1$.
We shall prove by iteration that for every $n\in \N$, there exists a constant $C_n$ such that 
\begin{equation}\label{eq:rec}
\forall x>0,\qquad u(x) \leq C_n \,x^{-\min(\sum_{k=1}^{n} \frac{1}{\beta^k},\, \frac{\alpha}{\beta}-\varepsilon ,\, 1-\frac{\varepsilon}{\beta^n})}.
\end{equation}
The two first terms in the minimum will appear naturally during the iterative step of the proof, the constant $\varepsilon$ being there to compensate for the slowly varying function $\ell_\alpha$. On the contrary, the last term is technical and has been added to ensure that the bound remains integrable near 0.\\
Let us start with the base case $n=1$.
Using that $F\circ u$ is decreasing and a change of variable, we have
$$F\left(u\left(\frac{1}{\lambda}\right)\right) \int_0^{1}e^{- z}dz  \leq   \int_0^{+\infty} e^{- z}F\left(u\left(\frac{z}{\lambda}\right)\right)  dz\leq \lambda \int_0^{+\infty} F(u(x))  dx.$$
Setting $\lambda= 1/x$ and applying Lemma \ref{lem:F}, we deduce that there exists a constant $C_1$ such that
$$ u(x)  \leq C_1 x^{-\frac{1}{\beta}} \leq C_1 x^{- \min\left(\frac{1}{\beta}, \frac{\alpha}{\beta}-\varepsilon, 1-\frac{\varepsilon}{\beta}\right)}, \qquad \text{as }x\rightarrow +\infty.$$
We now proceed to the induction step. Fix $n\in \N$ and assume that (\ref{eq:rec}) is satisfied. Since $F\circ u$ is decreasing, we have :
$$\Pb(L_\e>0)F(u(x))\leq \E\left[1_{\{L_\e>0\}}F(u(x-L_\e))\right] \leq  \E\left[1_{\{S_\e<x\}}F(u(x-L_\e))\right] + \Pb(S_\e\geq x) F(1).$$
Multiplying this inequality by $x$ and taking the Laplace transform, we obtain from Equation (\ref{eq:keyu})~:
\begin{multline*}
\Pb(L_\e>0)\int_0^{+\infty} e^{-\lambda x}x F(u(x)) dx \leq \int_0^{+\infty} e^{-\lambda x} x \left( \E\left[1_{\{S_\e<x\}}u(x-L_\e)\right]-u(x)   \right) dx \\+ (1+F(1)) \int_0^{+\infty} e^{-\lambda x} x  \Pb(S_\e\geq x)dx.
\end{multline*}
Using (\ref{eq:maxxf}) with $f=u$ yields 
$$\Pb(L_\e>0)\int_0^{+\infty} e^{-\lambda x}x F(u(x)) dx \leq  \E\left[S_\e e^{-\lambda S_\e}\right] \mathcal{L}[u](\lambda) + (1+F(1)) \int_0^{+\infty} e^{-\lambda x} x  \Pb(S_\e\geq x)dx.
$$
Then, using a change of variable and the recurrence assumption,
\begin{multline*}
\Pb(L_\e>0) F\left(u\left(\frac{1}{\lambda}\right)\right)    \int_0^{1}  e^{- z} z dz \leq  C_n \E[S_\e] \lambda  \int_0^{+\infty} e^{-z}   \left( \frac{z}{\lambda}\right)^{-\min\left(\sum_{k=1}^{n} \frac{1}{\beta^k}, \frac{\alpha}{\beta}-\varepsilon, 1-\frac{\varepsilon}{\beta^n}\right) }dz \\+  \lambda^2 (1+F(1)) \int_0^{+\infty} e^{-\lambda x} x  \Pb(S_\e\geq x)dx
\end{multline*}
i.e., there exists a constant $\widetilde{C}_n$ such that for $\lambda$ small enough,
$$u^\beta\left(\frac{1}{\lambda}\right)\leq \widetilde{C}_n  \left(\lambda^{1+\min\left(\sum_{k=1}^{n} \frac{1}{\beta^k}, \frac{\alpha}{\beta}-\varepsilon,  1-\frac{\varepsilon}{\beta^n}\right) } + \lambda^{\alpha} \ell_\alpha\left(\frac{1}{\lambda}\right)\right) .$$
As before, setting $\lambda = 1/x$, we conclude that for $x$ large enough
$$u\left(x\right) \leq \left(2\widetilde{C}_n\right)^{\frac{1}{\beta}} \left(x^{-\frac{1}{\beta}-  \frac{1}{\beta}\min\left(\sum_{k=1}^{n} \frac{1}{\beta^k},\frac{\alpha}{\beta}-\varepsilon, 1-\frac{\varepsilon}{\beta^n}\right)} + x^{-\frac{\alpha}{\beta} +\varepsilon}\right)  .$$
We now assume that $x\geq1$ and separate the different cases. 
\begin{enumerate}
\item On the one hand, if the minimum equals $\sum_{k=1}^{n} \frac{1}{\beta^k}$, we have 
$$u\left(x\right) \leq \left(2\widetilde{C}_n\right)^{\frac{1}{\beta}}  \left(x^{- \sum_{k=1}^{n+1} \frac{1}{\beta^k}} +   x^{-\frac{\alpha}{\beta}+\varepsilon} \right) \leq C_{n+1} x^{- \min\left(\sum_{k=1}^{n+1} \frac{1}{\beta^k}, \frac{\alpha}{\beta}-\varepsilon, 1-\frac{\varepsilon}{\beta^{n+1}}\right)}. $$
\item On the other hand, if the minimum equals $  \frac{\alpha}{\beta}-\varepsilon$, we obtain
$$u\left(x\right) \leq  \left(2\widetilde{C}_n\right)^{\frac{1}{\beta}}  \left(x^{- \frac{1-\varepsilon}{\beta}- \frac{\alpha}{\beta^2}} +   x^{-\frac{\alpha}{\beta}+\varepsilon} \right) \leq C_{n+1} x^{-\frac{\alpha}{\beta}+\varepsilon}  \leq C_{n+1} x^{-\min(\sum_{k=1}^{n+1} \frac{1}{\beta^k},\, \frac{\alpha}{\beta}-\varepsilon ,\, 1-\frac{\varepsilon}{\beta^{n+1}})}
$$
since $\varepsilon$ has been chosen small enough such that $1-\varepsilon+\alpha/\beta > \alpha$.
\item Finally, if the minimum equals $1-\frac{\varepsilon}{\beta^n}$,
$$u\left(x\right) \leq \widetilde{C}_{n+1}  \left(x^{- \ \frac{2}{\beta}+\frac{\varepsilon}{\beta^{n+1}}} +   x^{-\frac{\alpha}{\beta}+\varepsilon}  \right) \leq C_{n+1} x^{-\min(\sum_{k=1}^{n+1} \frac{1}{\beta^k},\, \frac{\alpha}{\beta}-\varepsilon ,\, 1-\frac{\varepsilon}{\beta^{n+1}})} $$
since $2/\beta >1$.
\end{enumerate}
This proves Formula (\ref{eq:rec}).
To conclude, note that 
$\sum_{k=1}^{+\infty} \frac{1}{\beta^k} = \frac{1}{\beta-1} > 1$
which implies that the sum may be removed from the minimum for $n$ large enough.
%
As a consequence, we conclude from the Tauberian theorem that
\begin{align*}
\limsup_{\lambda\rightarrow0} \frac{1}{\ell_\alpha(1/\lambda)}\lambda^{2-\alpha} \int_0^{+\infty} e^{-\lambda x} u(x)dx 
&\leq C_n \limsup_{\lambda\rightarrow0}  \frac{1}{\ell_\alpha(1/\lambda)}\lambda^{2-\alpha} \int_0^{+\infty} e^{-\lambda x}  x^{-\min(\frac{\alpha}{\beta}-\varepsilon ,\, 1-\frac{\varepsilon}{\beta^{n}})} dx=0
\end{align*}
since  $1-\alpha + \frac{\alpha}{\beta} >\varepsilon$ and  by taking $n$ large enough  $2-\alpha > \frac{\varepsilon}{\beta^n}$.\qed

\subsection{Back to the case $\alpha=1$}
It remains to finish the proof of the case $\alpha=1$. 
We first check that the conclusion of Lemma \ref{lem:key2} remains valid when we replace the condition (\ref{eq:assum2}) by
\begin{equation}\label{eq:Lfepsi}
\lim_{\lambda\downarrow0} \lambda^{1-\delta} \mathcal{L}[f](\lambda) = 0
\end{equation}
for some $\delta >0$. Indeed, since integrating by parts, we have
$$\E[S_\e e^{-\lambda S_\e}]=  \int_0^{+\infty} e^{-\lambda x} \Pb(S_\e \geq x)dx - \lambda \int_0^{+\infty} e^{-\lambda x} x \Pb(S_\e \geq x)dx \,\equi_{\lambda\downarrow0}\, -\ell_1 \ln(\lambda)$$
the lower bound (\ref{eq:lo}) of Lemma \ref{lem:key2} remains valid.
Then, going back to the upper bound (\ref{eq:up}), we deduce similarly that the first and third terms on the right-hand side go to 0.
For the second term, we have
$$1-\E\left[e^{-\lambda L_\e^+}\right] = \lambda \int_0^{+\infty} e^{-\lambda x} \Pb(L_\e^+ \geq x) dx \equi_{\lambda\downarrow0}\, -\ell_1 \lambda \ln(\lambda)$$
hence 
$$\limsup_{\lambda\downarrow0} \lambda^2 \ln(\lambda)  \mathcal{L}[xf](\lambda) \leq \limsup_{\lambda\downarrow0} \lambda\ln(\lambda)  \mathcal{L}[f](\lambda) =0$$
thanks to (\ref{eq:Lfepsi}). Finally, the last term on the right-hand side of (\ref{eq:up}) becomes :
$$
\lambda \E\left[1_{\{L_\e<0\}} \int_{0}^{-L_\e} e^{-\lambda z  } zf(z)dz\right] 
 \leq\lambda   A_\varepsilon^2\,f(0)  + \varepsilon C_\alpha  \lambda \int_{0}^{+\infty} e^{-\lambda z  } z\Pb(L_\e>z)dz \xrightarrow[\lambda\downarrow0]{} \varepsilon C_\alpha  \ell_1
$$
which proves that the conclusion of Lemma \ref{lem:key2} is still valid when $\alpha=1$, provided the stronger assumption (\ref{eq:Lfepsi}).
Now, assuming that $u$ and $F\circ u$ satisfy assumption (\ref{eq:Lfepsi}), we deduce as above that
$$\int_0^{x} z F(u(z))dz  \equi_{x\rightarrow+\infty} \ell_1 x $$
and the announced result follows from the monotone density theorem. It remains thus to check that 
$$ \lim_{\lambda\rightarrow0} \lambda^{1-\delta} \mathcal{L}[F\circ u](\lambda)=0  \qquad \text{ and }\qquad  \lim_{\lambda\rightarrow0} \lambda^{1-\delta}   \mathcal{L}[u](\lambda)=0.$$
The first asymptotics is an immediate consequence of the bound
$$
  \frac{1-\E\left[e^{-\lambda S_\e}\right]}{\lambda} \geq \Pb\left( L_\e\geq0,\, S_\e< 1\right)\int_1^{+\infty} e^{-\lambda x}   F(u(x)) dx. 
$$
For the second one, observe that since $F\circ u$ is decreasing, we deduce from Formula (\ref{eq:Fuln}) and Lemma \ref{lem:F}, that there exists a constant $K$ such that for $x>0$,
$$u(x) \leq K (\ln(x))^{1/\beta} x^{-\frac{1}{\beta}}.$$
As a consequence the Tauberian theorem (\ref{eq:taub}) yields
$$\limsup_{\lambda\downarrow0} \lambda^{1-\delta} \mathcal{L}[u](\lambda)  \leq  \lim_{\lambda\downarrow0} \lambda^{1-\delta} K  \int_0^{+\infty}e^{-\lambda x} (\ln(x))^{1/\beta} x^{-\frac{1}{\beta}} dx =0$$
since $\frac{1}{\beta}>\delta$.\qed

\end{document}